\documentclass[12pt]{article}
\usepackage{fullpage}
\usepackage{mathrsfs,pstricks,ifpdf,tikz}
\usepackage{amsfonts}
\usepackage{enumerate}

\usepackage{makecell}

\usepackage{latexsym}
\usepackage{amsmath,amssymb}
\usepackage{amsthm}
\usepackage{ifthen,calc}
\usepackage{color}
\usepackage{graphicx}
\usepackage{latexsym}
\usepackage{multirow}
\usepackage{diagbox}
\usepackage{float}
\usepackage[format=hang, margin=10pt]{caption}

\newtheorem{theorem}{Theorem}

\newtheorem{corollary}[theorem]{Corollary}

\newtheorem{lemma}[theorem]{Lemma}

\theoremstyle{definition}
\newtheorem{definition}[theorem]{Definition}

\long\def\skipit#1{}
\def\CL#1{\lceil#1\rceil}

\def\FR#1#2{\frac{#1}{#2}}
\def\wG{\widetilde G}
\def\wH{\widetilde H}
\def\NN{{\mathbb N}}

\begin{document}

\title{Upper bounds for bar visibility of subgraphs and $n$-vertex graphs}

\author{
Yuanrui Feng\thanks{Tianjin University, Tianjin, China: fyr666@tju.edu.cn.}\,
\qquad
Douglas B. West\thanks{Zhejiang Normal University, Jinhua, China, and
University of Illinois, Urbana, USA: dwest@math.uiuc.edu.
Supported by NNSF of China under Grant NSFC-11871439.}
\qquad
Yan Yang\thanks{Tianjin University, Tianjin, China: yanyang@tju.edu.cn.
Supported by NNSF of China under Grant NSFC-11401430.}\,
}

\date{\today}

\maketitle

\begin{abstract}
A {\it $t$-bar visibility representation} of a graph assigns each vertex up to
$t$ horizontal bars in the plane so that two vertices are adjacent if and only
if some bar for one vertex can see some bar for the other via an unobstructed
vertical channel of positive width.  The least $t$ such that $G$ has a $t$-bar
visibility representation is the {\it bar visibility number} of $G$, denoted by
$b(G)$.  We show that if $H$ is a spanning subgraph of $G$, then
$b(H)\le b(G)+1$.  It follows that $b(G)\le \CL{n/6}+1$ when $G$ is an
$n$-vertex graph.  This improves the upper bound obtained by Chang et al.\
(\emph{SIAM J. Discrete Math.} {\bf 18} (2004) 462).\\



\end{abstract}

\section{Introduction}
Various types of visibility representations of graphs have been studied for
their theoretical interests and applications in VLSI design.  In one type,
each vertex of the graph is assigned a horizontal bar in the plane.  The
assignment is a \textit{ bar visibility representation} if any two vertices are
adjacent if and only if the corresponding bars can see each other via an
unobstructed vertical channel of positive width.  A graph is a
\textit{ bar visibility graph} if it has a bar visibility representation.
Not all graphs have such representations; Tamassia and Tollis \cite{Tamassia86}
and Wismath \cite{Wismath85} gave a characterization:

\begin{theorem}[\cite{Tamassia86, Wismath85}]\label{3}
A graph admits a bar visibility representation if and only if it has a planar
embedding such that all cut-vertices lie on the boundary of one face.
\end{theorem}

As a generalization of bar visibility representation,
the $t$-bar visibility representation of a graph was introduced by
Chang, Hutchinson, Jacobson, Lehel and West \cite{Chang04} extended bar
representations to all graphs.  A \textit{ $t$-bar representation}
of a graph assigns each vertex of at most $t$ horizontal bars in the plane
so that two vertices $u$ and $v$ are adjacent if and only if some bar for $u$
and some bar for $v$ can see each other along an unobstructed vertical channel
of positive width.  The least $t$ such that $G$ has a $t$-bar visibility
representation is the \textit{bar visibility number} of $G$, denoted by $b(G)$.
Bar visibility graphs are the graphs $G$ such that $b(G)=1$.
Chang et al.\ studied the bar visibility number for planar graphs, the complete
graph $K_n$, the complete bipartite graph $K_{m,n}$, and graphs with $n$
vertices.

\begin{theorem}[\cite{Chang04}]\label{4}
Every planar graph has a $2$-bar visibility representation in which every
non-cut-vertex is assigned only one bar.
\end{theorem}

\begin{theorem}[\cite{Chang04}]\label{5}
For $n \geq 7$, $b(K_n) = \CL{n/6}$.
\end{theorem}

\begin{theorem}[\cite{Chang04}]\label{kmn}
For $m,n\in\NN$, $r\le b(K_{m,n})\le r+1$, where $r=\CL{\FR{mn+4}{2m+2n}}$.
\end{theorem}

\begin{theorem}[\cite{Chang04}]\label{6}
For any graph $G$ on $n$ vertices, $b(G) \leq \CL{n/6}+2$.
\end{theorem}

The upper bound in Theorem~\ref{kmn} has been improved by Cao, West, and
Yang~\cite{CWY} via a lengthy proof, yielding $b(K_{m,n})=r$ for all $(m,n)$.
In \cite{Chang04}, the authors conjectured that $b(G) \leq \CL{n/6}$ when $G$
has $n$ vertices; this must be restricted to $n\ge7$ due to $K_5$ and $K_6$.
We strengthen Theorem~\ref{6} toward the conjecture by giving a short proof
that $b(G)\le \CL{n/6}+1$ for every $n$-vertex graph $G$.

Our approach is to combine Theorem~\ref{5} with a bound on the bar visibility
number of subgraphs.  We prove $b(H)\leq b(G)+1$ whenever $H$ is a spanning
subgraph of $G$.  Equality can hold, but we do not know whether equality can
hold when $G=K_n$ with $n\ge7$.  Hence the conjecture about the extremal value
for $n$-vertex graphs remains open.

In fact, our technical result is a slightly stronger statement.
A {\it $t$-split} of a graph $G$ is a graph $G'$ in which each vertex is
replaced by a set of at most $t$ independent vertices in such a way that $u$
and $v$ are adjacent in $G$ if and only if some vertex in the set representing
$u$ is adjacent in $G'$ to some vertex in the set representing $v$.  The
{\it planar splitting thickness}  (or simply {\it split thickness}) of a graph
$G$, which we denote by $\sigma(G)$, is the minimum $t$ such that $G$ has a
$t$-split that is a planar graph.  The graph obtained from a $t$-bar
representation of $G$ by adding edges joining bars to record visibilities and
then shrinking each bar to a vertex is a $t$-split of $G$.  Hence
$\sigma(G)\le b(G)$.  We prove that always $b(H)\le \sigma(G)+1$ when $H$
is a spanning subgraph of $G$.  Note that setting $H=G$ yields
$\sigma(G)\le b(G)\le \sigma(G)+1$.  We have $b(G)=\sigma(G)$ if and only if
$G$ has a $\sigma(G)$-split in which all cut-vertices lie on one face.

The name ``planar splitting thickness'' was introduced in
Eppstein et al.~\cite{Epp}.  The concept was originally studied for complete
graphs by Heawood~\cite{Hea} in 1890, under the name ``$m$-pire problem''.
Ringel and Jackson~\cite{RJ} proved in effect that $K_n$ has a
$\CL{n/6}$-split.  A short proof of this by Wessel~\cite{Wes} was used
in~\cite{Chang04} to prove $b(K_n)=\CL{n/6}$ (Theorem~\ref{5} above).

\section{The transfer operation}

We introduce an operation that we will use to reduce the number of cut-vertices
in a graph.  Let $N_G(u)$ denote the neighborhood of a vertex $u$ in a graph
$G$, meaning the set of vertices adjacent to $u$.

\begin{definition}
Given a cut-vertex $u$ and another vertex $v$ in a graph $G$, the
{\it $u,v$-transfer} of $G$ is the graph $\wG$ obtained as follows: for every
$w\in N_G(u)$ such that $w$ is not in the component of $G-u$ containing $v$,
replace the edge $uw$ with the edge $vw$.
\end{definition}

A {\it $v$-lobe} of a graph $G$ is a subgraph of $G$ induced by the vertex set
consisting of $v$ and a component of $G-v$.  If $v$ is a vertex in a graph $G$,
then $G$ is planar if and only if every $v$-lobe of $G$ is planar.  This uses
the fact that if $F$ is the set of edges bounding some face in a planar
embedding of a graph $G$, then $G$ has an embedding in which $F$ is the set of
edges incident to the unbounded face.

\begin{lemma}\label{1}
Let $u$ be a cut-vertex in a planar graph $G$.  If $\wG$ is the
$u,v$-transfer of $G$, where $v\in V(G)-\{u\}$, then $\wG$ is planar,
$u$ is not a cut-vertex of $\wG$, and a vertex outside $\{u,v\}$ is a
cut-vertex of $\wG$ if and only if it is a cut-vertex of $G$.
\end{lemma}


\begin{proof}
Paths in $G$ correspond to paths through the same vertices in $\wG$, with
the possible substitution of $v$ for $u$, especially when the paths join
vertices from different $u$-lobes.  Hence the $v$-lobes of $\wG$ are isomorphic
to the $u$-lobes of $G$.  Since $G$ is planar, its $u$-lobes are planar, and
then their combination as $v$-lobes makes $\wG$ planar.

Vertex $u$ is not a cut-vertex of $\wG$, since in $\wG$ all its neighbors are
in the component of $G-u$ containing $v$, which remains a connected subgraph
of $G$.  Note that if $v$ and $u$ are not in the same component of $G$,
then $u$ is isolated in $\wG$.  Also, $v$ is a cut-vertex of $\wG$.

A vertex $w$ is a cut-vertex if and only if its graph has more than one
$w$-lobe.  If $w$ is a cut-vertex in $G$ outside $\{u,v\}$, then the $w$-lobes
of $G$ correspond to $w$-lobes in $\wG$, except that the $w$-lobe containing
$u$ in $G$ may differ from the $w$-lobe containing $v$ in $\wG$.  In
particular, there is more than one $w$-lobe in $G$ if and only if there is
more than one $w$-lobe in $\wG$.
\end{proof}

\section{Application to visibility number}

We note first that it is possible for a graph to have smaller bar visibility
number than a spanning subgraph.  Let $H$ be the graph obtained from $K_{3,3}$
by replacing one edge with a path of length $3$ and then deleting the middle
edge of the path (on the left in Figure~1).  Note that $H$ is planar, but it
has no planar embedding where the cut-vertices lie on the same face.
Letting $C$ be the $4$-cycle $[a,b,g,h]$, the claim holds because deleting the
vertices of $C$ leaves the cut-vertices of $H$ in components whose vertices of
attachment to $C$ alternate along $C$.  Hence one must be embedded inside and
one outside.  Thus $b(H)\ge2$; in fact, equality holds.  On the other hand, we
can obtain a planar supergraph $G$ having no cut-vertices by adding one edge at
each vertex of degree $1$ (on the right in Figure~1).  Hence $b(H)=b(G)+1$.

\begin{figure}[htbp]
\begin{center}
\begin{tikzpicture}[scale = 1.00]
	\begin{scope}[xshift = -6cm]
	\coordinate (a) at (-0.6,0.8);		\coordinate (b) at (0.6,0.8);
	\coordinate (c) at (-0.4,0);		\coordinate (d) at (-1.2,0);
	\coordinate (e) at (1.2,0);		\coordinate (f) at (0.4,0);
	\coordinate (g) at (-0.6,-0.8);		\coordinate (h) at (0.6,-0.8);

	\node at (-0.75,0.97) {$a$};		\node at (0.75,1.00) {$b$};
	\node at (-.35,-0.2) {$c$};		\node at (-1.4,0.05) {$d$};
	\node at (1.4,0) {$e$}; 		\node at (0.28,0.16) {$f$};
	\node at (-0.55,-1.07) {$g$};		\node at (0.73,-1) {$h$};
	
	\filldraw[thick, fill = black]{
		(a) circle(1.4pt) 			(b) circle(1.4pt)
		(c) circle(1.4pt)			(d) circle(1.4pt)
		(e) circle(1.4pt)			(f) circle(1.4pt)
		(g) circle(1.4pt)			(h) circle(1.4pt)
	};

	\draw{
			(d) -- (a) -- (b) -- (e) -- (h) -- (g) -- cycle
		 	(c) -- (d)		(e) -- (f)
			(a) -- (h)
	};
	\draw (b) .. controls (0.6, 2.2) and (-2.4, 1.8) .. (-2.4,0) .. controls (-2.4, -0.8) and (-1.2, -1.2) .. (g);
	\end{scope}

	\begin{scope}
	\coordinate (a) at (-0.6,0.8);		\coordinate (b) at (0.6,0.8);
	\coordinate (c) at (-0.4,0);		\coordinate (d) at (-1.2,0);
	\coordinate (e) at (1.2,0);		\coordinate (f) at (0.4,0);
	\coordinate (g) at (-0.6,-0.8);		\coordinate (h) at (0.6,-0.8);

	\node at (-0.75,0.97) {$a$};		\node at (0.75,1.00) {$b$};
	\node at (-.35,-0.2) {$c$};		\node at (-1.4,0.05) {$d$};
	\node at (1.4,0) {$e$}; 		\node at (0.28,0.16) {$f$};
	\node at (-0.55,-1.07) {$g$};		\node at (0.73,-1) {$h$};
	
	\filldraw[thick, fill = black]{
		(a) circle(1.4pt) 			(b) circle(1.4pt)
		(c) circle(1.4pt)			(d) circle(1.4pt)
		(e) circle(1.4pt)			(f) circle(1.4pt)
		(g) circle(1.4pt)			(h) circle(1.4pt)
	};

	\draw{
			(d) -- (a) -- (b) -- (e) -- (h) -- (g) -- cycle
		 	(c) -- (d)		(e) -- (f)
			(a) -- (h)
			(c) -- (a)		(h) -- (f)
	};
	\draw (b) .. controls (0.6, 2.2) and (-2.4, 1.8) .. (-2.4,0) .. controls (-2.4, -0.8) and (-1.2, -1.2) .. (g);
	\end{scope}
\end{tikzpicture}
\end{center}

\vspace{-1pc}
\caption{Graphs $H$ and $G$.}
\label{fig:3}
\end{figure}
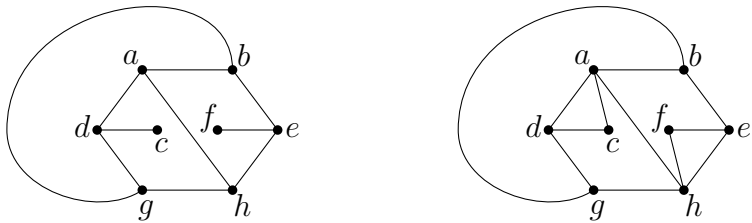

We do not know whether for $k>1$ there is a graph $G$ with $b(G)=k$ having a
spanning subgraph $H$ with $b(H)=k+1$.


Recall that $\sigma(G)$ is the split thickness of $G$.

\begin{theorem}\label{2}
If $H$ is a spanning subgraph of a graph $G$, then $b(H) \leq \sigma(G)+1$.
\end{theorem}

\begin{proof}
Let $t=\sigma(G)$, and let Let $G'$ be a planar $t$-split of $G$.  Each vertex
$u\in V(G)$ is represented in $G'$ by an independent set $S(u)$ of size at most
$t$ such that vertices are adjacent in $G$ if and only if some vertices
representing them are adjacent in $G'$.

Let $H$ be a spanning subgraph of $G$, and let $H'$ be the spanning subgraph
of $G'$ obtained from $G'$ by deleting all edges of $G'$ whose endpoints
represent vertices that are not adjacent in $H$.  If $H'$ has no cut-vertices,
then $H'$ is a bar visibility graph, and a bar visibility representation of
$H'$ yields a $t$-bar representation of $H$.

However, $H'$ may have cut-vertices.  If $u_1$ and $u_2$ are both cut-vertices
of $H'$ that represent $u$, perform the $u_1,u_2$-transfer operation in $H'$.
The resulting graph $\wH$ is planar and has one less cut-vertex than $H'$; in 
particular, fewer vertices in $S(u)$ are cut-vertices.  Furthermore, $\wH$
is a planar $t$-split of $H$; it is planar by Lemma~\ref{1}, and edges that
previously joined representatives of two vertices have been replaced by edges
that join representatives of the same two vertices.

Continue such transfer operations to obtain a final planar $t$-split $H^*$
of $H$.  Although $H^*$ may have cut-vertices, at most one vertex in each set
$S(u)$ is a cut-vertex.  Now apply Theorem~\ref{4} to $H^*$.  We obtain a 
$2$-bar visibility representation of $H^*$ in which the only vertices assigned
two bars are cut-vertices, and there is at most one of these in each set
$S(u)$.  Together, the representatives of $u$ are assigned at most $t+1$ bars.
Combining the bars assigned to representatives of $u$, for each $u$, yields
a $(t+1)$-bar representation of $H$.
\end{proof}

\begin{corollary}
For any graph $G$ on $n$ vertices, $b(G) \leq \CL{n/6}+1$.
\end{corollary}

\begin{proof}
For $n\leq 4$, it is easy to check that all $n$-vertex graphs are bar
visibility graphs, with bar visibility number $1$.

Deleting an edge from $K_5$ leaves a $2$-connected planar graph.  Hence
$K_5$ decomposes into two planar graphs without cut-vertices, and therefore
$b(K_5)\le 2$.  Every proper subgraph $G$ of $K_5$ is planar, and hence by
Theorem \ref{4} we have $b(G)\leq 2$.

Bonamy and Perrett~\cite{Bonamy16} showed that every $6$-vertex graph 
decomposes into three paths.  Chang et al.~\cite{Chang04} showed that if
a $k$-vertex graph $G$ decomposes into at most $\lceil k/2 \rceil$ paths, then
$b(G)\leq \lceil k/6 \rceil+1$.  Together, these statements yield
$b(G)\leq 2$ when $G$ has six vertices.

For $n\geq 7$, Theorems \ref{2} and \ref{5} and $\sigma(K_n)\le b(K_n)$ yield
$b(G) \leq \CL{n/6}+1$ when $G$ has $n$ vertices.
\end{proof}

For $n\ge7$, the conjecture remains open that the bound can be
improved to $\CL{n/6}$.

\end{document}